\documentclass[11pt,reqno]{amsart}
\usepackage{amssymb,eucal,latexsym}

\usepackage{color}
\usepackage{hyperref}

\newcommand{\dnw}{\mathbin{\downarrow}}
\newcommand{\upw}{\mathbin{\uparrow}}

\numberwithin{equation}{subsection}

\theoremstyle{plain}
\newtheorem{lemma}{Lemma}
\newtheorem{theorem}[lemma]{Theorem}
\newtheorem{corollary}[lemma]{Corollary}

\newtheorem{proposition}[lemma]{Proposition}

\theoremstyle{definition}
\newtheorem{definition}[lemma]{Definition}
\newtheorem{example}[lemma]{Example}

\begin{document}
\sloppy

\author[N.\,A. Bazhenov]{Nikolay A. Bazhenov}
\address[N.\,A. Bazhenov]{Novosibirsk State University\\Pirogova str. 1\\630090 Novosibirsk\\Russia}
\email{nickbazh@yandex.ru}

\author[I.\,Sh. Kalimullin]{Iskander Sh. Kalimullin}
\address[I.\,Sh. Kalimullin]{Kazan (Volga Region) Federal University\\Kremlevskaya str. 18\\420008 Kazan\\Russia}
\address[I.\,Sh. Kalimullin]{Novosibirsk State University\\Pirogova str. 1\\630090 Novosibirsk\\Russia}
\email{Iskander.Kalimullin@kpfu.ru}

\author[M.\,V. Schwidefsky]{Marina V. Schwidefsky}
\address[M.\,V. Schwidefsky]{Novosibirsk State University\\Pirogova str. 1\\630090 Novosibirsk\\Russia}
\address[M.\,V. Schwidefsky]{Sobolev Institute of Mathematics SB RAS\\ Acad. Koptyug ave. 4\\630090 Novosibirsk\\Russia}
\email{m.schwidefsky@g.nsu.ru}

\thanks{The work is supported by the Mathematical Center in Akademgorodok under the agreement No.~075-15-2025-349 with the Ministry of Science and Higher Education of the Russian Federation.}

\title[Effective Stone duality]{\sc An effective version of the Stone duality}


\date{\today}

\begin{abstract}
	The paper studies computability-theoretic aspects of topological $T_0$-spaces. We introduce effective versions of the notions of a countable $c$-poset and a (second-countable) topological space with base. Based on this, we prove an effective version of the known Stone-type duality between the category $\mathbf{AS}$ (whose objects are almost semispectral spaces with base and whose morphisms are spectral mappings) and the category $\mathbf{DP}$ (whose objects are distributive $c$-posets and whose morphisms are strict mappings). Namely, we show that for an arbitrary set $Z\subseteq \omega$, this duality is preserved when one restricts to objects which have $Z$-computably enumerable presentations only. Following this approach, we establish several results in computable topology. We prove that every degree spectrum of a countable algebraic structure can be realized as the degree spectrum of a topological space with base. We show that for any non-zero natural number $N$, there is a computable topological space with base that has precisely $N$-many computable copies, up to effective spectral homeomorphisms.
\end{abstract}	

\maketitle

\thispagestyle{empty}

\section*{Introduction}

\noindent
In 1937, M.\,H. Stone \cite{Stone2} established the dual equivalence of the category of distributive lattices with $0$ and $1$ (distributive $(0,1)$-lattices) with
lattice homomorphisms preserving the two constants, $0$ and $1$, ($(0,1)$-lattice homomorphisms) as morphisms and the category of spectral spaces with spectral maps as morphisms. This result extended an earlier result of his \cite{Stone1} on the dual equivalence of the category of Boolean algebras with homomorphisms and the category of Stone spaces (that is, compact totally disconnected Hausdorff spaces) with continuous maps.
Later in 1970, H.\,A. Priestley \cite{P1,P2} established another type of duality (nowadays, generally referred to as the Priestley duality).

Since that time, there was a fair amount of different successful attempts to extend the Stone duality to a wider class of partially ordered sets than distributive bounded lattices. It is very hard to present here a complete account on the papers dedicated to this topic. We mention here just two of them---the paper \cite{GJ} by L.\,J. Gonz\'{a}lez and R. Jansana and the paper \cite{CG} by S.\,A. Celani and L.\,J. Gonz\'{a}lez. In the first one, a Stone-type duality was established for the category
of all posets, while in the second one, Stone-type dualities were found for distributive meet-semilattices and lattices (not necessarily distributive).

In \cite{ES1,ES2} a representation theorem for the so-called distributive $c$-posets was proved.
In \cite{S1,S2}, this representation theorem was extended to a Stone-type duality which,
in its turn, is an extension of the Stone duality for distributive $(0,1)$-lattices was made to the category of the distributive $c$-posets, see Definition
\ref{DP} below. It is a curious fact that the duality result of L.\,J. Gonz\'{a}lez and R. Jansana \cite{GJ} is a particular case of the duality
from \cite{S1,S2}---this is demonstrated in M.\,M. Mingott Fernandez \cite{MF}.

Another source of our motivation is provided by the developments in computable analysis and topology (we refer to the monographs~\cite{Wei00,DowneyMelnikov} for the background in this area). On the one hand, recent works gained significant insight into effective content of the classical Stone duality for Boolean algebras. For example, M. Harrison-Trainor, A. Melnikov, and K.\,M. Ng \cite{HMN-20} proved the following result: a countable Boolean algebra $\mathcal{B}$ is isomorphic to a computable algebraic structure if and only if its dual Stone space $\operatorname{St}(\mathcal{B})$ is homeomorphic to a computable Polish space. For further results on Stone spaces (viewed as effectively presented Polish spaces) we refer to, e.g., \cite{HMN-20,HKS-ta,BHM-23,MN-23}. We note that these works heavily employ techniques from computable structure theory~\cite{DowneyMelnikov}, in particular, the notion of the degree spectrum of an algebraic structure, see, e.g., \cite{HKS-ta}.

On the other hand, there are several known definitions of the following notion: which second-countable topological spaces could be called computable---we refer to, e.g., \cite{KW-85,Spreen-90,GW}. We note that recent papers~\cite{BMN-24,MNH-ta} have studied computable topological presentations for Polish spaces. M. Hoyrup, A. Melnikov, and K.\,M. Ng~\cite{MNH-ta} proved that every countably-based $T_0$-space has a computable topological presentation in the sense of~\cite{GW}. We refer to, e.g., \cite{BR-ta} for further results on computable topological spaces. Following this line of research, here we also focus on the following question: which notion of a computable topological space is suitable for establishing effective versions of Stone-type dualities of~\cite{ES1,ES2,S1}?

In this paper, we give an effective version of the Stone duality from \cite{S1,S2}.
Some of the results presented here were announced in \cite{BKS-AL}.

The paper is organized as follows. Section~\ref{sect:prelim} discusses the preliminaries. Section~\ref{sect:spectra} contains the necessary results on spectra of posets. In particular, Theorem~\ref{T:Stone} formulates the main Stone-type duality proved in~\cite{S1}---the duality between the category $\mathbf{AS}$ (whose objects are almost semispectral spaces with base and whose morphisms are spectral mappings) and the category $\mathbf{DP}$ (whose objects are distributive $c$-posets and whose morphisms are strict mappings). 

Section~\ref{sect:gen-obj} gives the notions of effective presentations for objects from the categories $\mathbf{AS}$ and $\mathbf{DP}$. 
For a given set $Z\subseteq \omega$, we define a $Z$-computably enumerable (or $Z$-c.e., for short) $c$-poset and a $Z$-computably enumerable space with base (Definitions~\ref{def:comp-c-poset} and~\ref{def:comp-top-space-with-base}). Theorem~\ref{prop:complexity-presentations} and Corollary~\ref{C:Stone-Z} establish an effective version of Theorem~\ref{T:Stone}: a $c$-poset $\mathcal{P}$ from $\mathbf{DP}$ has a $Z$-c.e.\ presentation if and only if its dual topological space with base $\mathsf{T}(\mathcal{P})$ from $\mathbf{AS}$ has a $Z$-c.e.\ presentation. 

In Definition~\ref{def:Z-comp-things}, we additionally define \emph{$Z$-computable} $c$-posets and spaces with base. Following the approach of~\cite{HMN-20,HKS-ta}, we introduce the notion of the degree spectrum for a space with base $\mathbb{X}$ (Definition~\ref{def:spectra-of-spaces}), and we prove that every degree spectrum of a countable algebraic structure $\mathcal{S}$ can be realized as the degree spectrum of an appropriately chosen space with base $\mathbb{X}_{\mathcal{S}}$ (Theorem~\ref{theo:degree-spectra}). We note that, to our best knowledge, it is still unknown whether an analogue of Theorem~\ref{theo:degree-spectra} holds for degree spectra of Polish spaces $\mathbb{Y}$ (up to homeomorphism).

Section~\ref{sect:semilat} considers some familiar subcategories of the category $\mathbf{DP}$: for example, the category $\mathbf{DL}$ (whose objects are distributive lattices and whose morphisms are strict lattice homomorphisms) which is dual to the category $\mathbf{ASpec}$ (whose objects are almost spectral spaces and whose morphisms are spectral mappings). Among other things, Theorem~\ref{theo:semi_lattices} proves that a distributive lattice $\mathcal{D}$ is isomorphic to a $Z$-computable (in the usual sense of computable structure theory) lattice if and only if the dual $\mathbf{ASpec}$-space $\mathsf{T}(\mathcal{D})$ has a $Z$-computable presentation. 

We also introduce effective versions for the notions of morphisms (Definition~\ref{def:Z-comp-spectral}). We show that for a natural number $N\geq 1$, there is a computable topological space with base that has precisely $N$-many computable copies, up to effective spectral homeomorphisms (Corollary~\ref{cor:comp-dim}).

\section{Preliminaries}\label{sect:prelim}

\noindent
Before presenting the Stone-type dualities obtained in~\cite{S1,S2}, here we discuss the necessary background on posets and topological spaces. 

Suppose that $\mathbf{Cat}$ is a category, and $X$ is an object from $\mathbf{Cat}$. We say that an object $Y$ from $\mathbf{Cat}$ is a \emph{copy} (or a presentation) of $X$ if $Y$ is $\mathbf{Cat}$-isomorphic to $X$.

\medskip

\subsection{Distributive posets}
{\ }

A closure operator $\varphi$ on a set $P$ is \emph{algebraic} if $\varphi(A)=\bigcup_{F\subseteq_{fin}A}\varphi(F)$ for each $A\subseteq P$.

\begin{definition}\label{DP}\cite{BS}
A $c$-\emph{poset} is a structure $\mathcal{P}=\langle P;\leq,\varphi\rangle$ such that:
\begin{enumerate}
\item[(i)]
$\langle P;\leq\rangle$ is a poset;
\item[(ii)]
$\varphi$ is an algebraic closure operator on $P$ such that $\varphi(\varnothing)=\varnothing$ and, for all $x,y\in P$,
\[
x\leq y\quad\text{if and only if}\quad\varphi(x)\subseteq\varphi(y).
\]
\end{enumerate}
A $c$-poset $\mathcal{P}=\langle P;\leq,\varphi\rangle$ is \emph{distributive} if the lattice $\operatorname{Id}\mathcal{P}$ of $\varphi$-closed subsets of $P$ is distributive.

Each $\varphi$-closed subset of $P$ is called a $\varphi$-\emph{ideal} of $\langle P;\leq\rangle$ or just an \emph{ideal} of $\mathcal{P}$.
An ideal $I\in\operatorname{Id}\mathcal{P}$ is \emph{proper} if $I\notin\{\varnothing,P\}$.
A set $F\subseteq P$ is a \emph{filter} of $\langle P;\leq\rangle$ if it is a down-directed upper cone with respect to $\leq$.
\end{definition}

\begin{example}
For a poset $\langle P;\leq\rangle$ and for a set $A\subseteq P$, let $\dnw A$ denote the \emph{lower cone} in $\langle P;\leq\rangle$ generated by $A$;
that is, $\dnw A=\{p\in P\mid p\leq a\ \text{for some}\ a\in A\}$.

It is clear that $\langle P;\leq,\varphi\rangle$ with $\varphi(A)=\dnw A$ for $A\subseteq P$ is a distributive $c$-poset.
\end{example}

\noindent
For a subset $X$ in a poset $\mathcal{P}=\langle P;\leq\rangle$, we denote the set of all lower bounds for $X$ in $\mathcal{P}$ by $L(X)$
and the set of all upper bounds for $X$ in $\mathcal{P}$ by $U(X)$. Then we have for each subset $X\subseteq P$:
\[
\bigcap_{x\in X}\varphi(x)=L(X)\in\operatorname{Id}\mathcal{P}.
\]

\begin{lemma}\label{L:BS}\cite[Proposition 3.1]{BS},
For a $c$-poset $\mathcal{P}=\langle P;\leq,\varphi\rangle$ and for a proper ideal $I$ of $\mathcal{P}$, the following conditions are equivalent.
\begin{enumerate}
\item
The set $P{\setminus}I$ is a filter of $\langle P;\leq\rangle$.
\item
$I$ is a $\cap$-prime element of the closure lattice $\operatorname{Id}\mathcal{P}$.
\item
Inclusion $L(a_0,a_1)\subseteq I$ implies that $a_i\in I$ for some $i<2$.
\end{enumerate}
\end{lemma}

\begin{definition}\label{Id}\cite{BS}
Let $\mathcal{P}=\langle P;\leq,\varphi\rangle$ be a $c$-poset.
A proper ideal $I$ of $\mathcal{P}$ is \emph{prime} if it satisfies one of the equivalent statements of Lemma \ref{L:BS}.
By $\operatorname{Spec}\mathcal{P}$, we denote the set of all prime ideals of $\mathcal{P}$.
\end{definition}

\begin{theorem}\label{BS}\cite[Theorem 3.3]{BS}
Let $\mathcal{P}=\langle P;\leq,\varphi\rangle$ be a distributive $c$-poset, let $I\subseteq P$ be a nonempty ideal of $\mathcal{P}$ and let $F\subseteq P$
be a nonempty down-directed set such that $I\cap F=\varnothing$. Then there is a prime ideal $Q\subseteq P$ such that $I\subseteq Q$ and $Q\cap F=\varnothing$.
\end{theorem}

\begin{definition}\label{D:Pm}\cite{S1}
Let $\mathbf{DP}$ denote the category whose objects are distributive $c$-posets and whose morphisms are mappings $f\colon P_0\to P_1$,
where $\mathcal{P}_0=\langle P_0;\leq,\varphi_0\rangle$ and $\mathcal{P}_1=\langle P_1;\leq,\varphi_1\rangle$ are distributive $c$-posets,
which satisfy the following condition:
\begin{enumerate}
\item[--]
$f$ is \emph{strict}; that is, $f^{-1}(I)$ is a prime ideal of $\mathcal{P}_0$ for each prime ideal $I$ of $\mathcal{P}_1$.
\end{enumerate}
If a surjection $g\colon P_0\to P_1$ satisfies the following properties:
\begin{enumerate}
\item[--]
$a\leq b$ in $\mathcal{P}_0$ if and only if $g(a)\leq g(b)$ in $\mathcal{P}_1$ for all $a,b\in P_0$;
\item[--]
$g\bigl(\varphi_0(X)\bigr)=\varphi_1\bigl(g(X)\bigr)$ for all $X\subseteq P_0$,
\end{enumerate}
then $g$ is called an \emph{isomorphism} of $c$-posets.
\end{definition}

\medskip

\subsection{Topological spaces with base}{\ }

For all the topological notions which we do not define here, we refer to the monograph of Yu.\,L. Ershov \cite{Top} as well as to the one of
G. Gierz \emph{et al.} \cite{Comp}.

A topological $T_0$-space $\mathbb{X}$ is [\emph{almost}] \emph{sober} if, for each closed [proper] nonempty set $F\subseteq X$, there is $x\in X$ such that $F=\dnw x$.
We note that, whenever we speak of a partial order in a topological $T_0$-space, we always mean the \emph{specialization order}.

\begin{definition}\label{D:B}
\cite{S1,S2}
A triple $\mathbb{X}=\langle X,\mathcal{T},\mathcal{B}\rangle$ is a \emph{topological space with base} or just a \emph{space with base}, if
the following conditions are satisfied:
\begin{enumerate}
\item[(i)]
$\langle X,\mathcal{T}\rangle$ is a topological $T_0$-space and $\mathcal{B}$ forms a basis of the topology $\mathcal{T}$;
\item[(ii)]
$\varnothing\in\mathcal{B}$ if and only if $\mathbb{X}$ is sober and the poset $\langle\mathcal{B};\subseteq\rangle$ is down-directed;
\item[(iii)]
$X\in\mathcal{B}$ if and only if $\mathbb{X}$ is compact and the poset $\langle\mathcal{B};\subseteq\rangle$ is up-directed.
\end{enumerate}
We write $\mathcal{T}(\mathbb{X})$ for $\mathcal{T}$ and $\mathcal{B}(\mathbb{X})$ for $\mathcal{B}$ in this case.

We say that $\mathbb{X}$ is a [\emph{topological}] \emph{space with up-directed base} [\emph{down-directed base}] if the poset
$\langle\mathcal{B};\subseteq\rangle$ is up-directed [down-directed];
$\mathbb{X}$ is a \emph{space with} $1$-\emph{base} [$0$-\emph{base}] if
$\langle\mathcal{B};\subseteq\rangle$ has a greatest element [a least element].
Moreover, $\mathbb{X}$ is a \emph{space with multiplicative base} if $\mathcal{B}$ is closed under finite nonempty intersections;
$\mathbb{X}$ is a \emph{space with additive base} if $\mathcal{B}$ is closed under finite nonempty unions.
\end{definition}

\begin{definition}\label{D:pf-st}
A topological space with base $\mathbb{X}$ is called an \emph{almost semispectral space with base}, if $\langle\mathbb{X},\mathcal{T}(\mathbb{X})\rangle$ is an almost sober space,
and $\mathcal{B}(\mathbb{X})$ consists of open compact sets.
\end{definition}

\begin{definition}\label{D:AS}
Let $\mathbf{AS}$ be the category whose objects are almost semispectral spaces with base and whose morphisms are \emph{spectral} mappings; that is,
mappings which satisfy the following condition:
\begin{enumerate}
\item[]
if $f\colon\mathbb{X}\to\mathbb{Y}$, where $\mathbb{X},\mathbb{Y}$ are almost semispectral spaces with base, then
$f^{-1}(U)\in\mathcal{B}(\mathbb{X})$ for each $U\in\mathcal{B}(\mathbb{Y})$.
\end{enumerate}
\end{definition}

\begin{lemma}\label{01}\cite[Lemma 3]{S2}
Let $\mathbb{X}$ be an almost sober topological space with base.
\begin{enumerate}
\item[(i)]
$\mathbb{X}$ is a space with $0$-base if and only if $\varnothing\in\mathcal{B}(\mathbb{X})$.
In particular, $\mathbb{X}$ is a space with $0$-base if and only if $\mathbb{X}$ is a sober space with down-directed base.
\item[(ii)]
$\mathbb{X}$ is a space with $1$-base if and only if $X\in\mathcal{B}(\mathbb{X})$.
In particular, $\mathbb{X}$ is a space with $1$-base if and only if $\mathbb{X}$ is a compact space with up-directed base.
\end{enumerate}
\end{lemma}

\noindent

\section{Spectra of posets}\label{sect:spectra}

Here we give some technical results which are used in the proof of the duality from Theorem~\ref{T:Stone} given below. These results will be also useful for effective versions of Theorem~\ref{T:Stone}.

\begin{definition}\label{D:SpL}\cite{ES1,ES2,Top}
Let $\mathcal{P}=\langle P;\leq,\varphi\rangle$ be a $c$-poset and let
$\operatorname{Spec}\mathcal{P}$ denote the set of all prime ideals of $\mathcal{P}$. For each $a\in P$, we put
\[
V_a=\{I\in\operatorname{Spec}\mathcal{P}\mid a\notin I\}.
\]
The space $\operatorname{\mathbb{S}pec}\mathcal{P}=\langle\operatorname{Spec}\mathcal{P},\mathcal{T},\mathcal{B}\rangle$, where
$\mathcal{T}$ denotes the topology with the (sub)basis $\mathcal{B}=\{V_a\mid a\in P\}$, is called the \emph{spectrum of} $\mathcal{P}$.

The space $\operatorname{\mathbb{S}pec}\mathcal{S}$ is called the
\emph{spectrum of a join-semilattice} $\langle S;\vee\rangle$, where $\mathcal{S}=\langle S;\vee,\psi\rangle$ and $\psi(X)$ denotes the join-semilattice ideal
generated by $X\subseteq S$; that is,
\[
\psi(X)=\{s\in S\mid s\leq a_0\vee\ldots\vee a_n\ \text{for some}\ n<\omega\ \text{and some}\ a_0,\ldots,a_n\in X\}.
\]
The \emph{spectrum of a lattice} $\langle L;\vee,\wedge\rangle$ is the spectrum of its join-semilattice reduct $\langle L;\vee\rangle$.
\end{definition}

\noindent
Let $\mathcal{P}=\langle P;\leq,\varphi\rangle$ be a $c$-poset. For a set $X\subseteq P$, we put
\[
V_X=\{I\in\operatorname{Spec}\mathcal{P}\mid X\nsubseteq I\}.
\]
It is clear that $V_X=\bigcup_{a\in X}V_a$.

\begin{lemma}\label{L:phi}
For a $\varphi$-distributive $c$-poset $\mathcal{P}=\langle P;\leq,\varphi\rangle$, the following statements hold.
\begin{enumerate}
\item[(i)]
$V_X=V_{\varphi(X)}$ for each set $X\subseteq P$.
\item[(ii)]
$V_a\subseteq V_X$ if and only if $a\in\varphi(X)$ for each element $a\in P$ and each set $X\subseteq P$.
\item[(iii)]
$V_a\subseteq V_b$ if and only if $a\leq b$ in $\mathcal{P}$ for all $a,b\in P$.
\item[(iv)]
$V_a\cap V_b=V_c$ if and only if $a\wedge b=c$ in $\mathcal{P}$ for all $a,b,c\in P$.
\item[(v)]
Suppose that $a\vee b\in\varphi(a,b)$ whenever $a\vee b$ exists.
Then $V_a\cup V_b=V_c$ if and only if $a\vee b=c$ in $\mathcal{P}$ for all $a,b,c\in P$.
\end{enumerate}
\end{lemma}

\begin{proof}
Statement (i) is the content of \cite[Corollary 6(iii)]{ES1}.

(ii) If $a\in\varphi(X)$, then $V_a\subseteq V_{\varphi(X)}$ by (i).
Suppose that $a\notin\varphi(X)$ for some $a\in P$ and some $X\subseteq P$. In this case, $\varphi(X)\cap\upw a=\varnothing$.
Applying Theorem \ref{BS}, we obtain a $\varphi$-ideal $I\in\operatorname{Spec}\mathcal{P}$ such that $\varphi(X)\subseteq I$ and $a\notin I$.
This implies in view of (i) that $I\in V_a{\setminus}V_X$. Thus, $V_a\nsubseteq V_X$ and (ii) follows.

Statement (iii) follows from (ii).

(iv) Suppose first that $a\wedge b$ exists. It follows from (iii) that $V_{a\wedge b}\subseteq V_a\cap V_b$.
Conversely, let $I\in V_a\cap V_b$; then $a,b\notin I$. This means by Lemma \ref{L:BS} that $a\wedge b\notin I$ and $I\in V_{a\wedge b}$.
Hence, $V_{a\wedge b}=V_a\cap V_b$. Suppose now that $V_a\cap V_b=V_c$; then $c\leq a,b$ by (iii). If $d\leq a,b$ for some $d\in P$ then
we obtain by (iii) that $V_d\subseteq V_a\cap V_b=V_c$ whence $d\leq c$. This proves that $c$ is a greatest lower bound for $a,b$; that is, $c=a\wedge b$.

(v) As in the proof of (iv), we suppose first that $a\vee b$ exists. It follows from (iii) that $V_a\cup V_b\subseteq V_{a\vee b}$.
Conversely, let $I\in V_{a\vee b}$; then $a\vee b\notin I$. By our assumption about $\varphi$, this implies that $a\notin I$ or $b\notin I$
whence $I\in V_a\cup V_b$ and $V_{a\vee b}=V_a\cup V_b$.
We obtain by (iii) that $V_c=V_a\cup V_b\subseteq V_d$ whence $c\leq d$.
This proves that $c$ is a least upper bound for $a,b$; that is, $c=a\vee b$.
\end{proof}

\noindent
Let $\mathbb{X}$ be a topological $T_0$-space and let $\mathcal{F}\subseteq\mathcal{T}(\mathbb{X})$ be a family of open sets.
Define a closure operator $\varphi_\mathcal{F}$ on $\mathcal{F}$ as follows. If $\mathcal{X}\subseteq\mathcal{F}$ then we put
\[
\varphi_\mathcal{F}(\mathcal{X})=
\begin{cases}
&\varnothing,\quad\text{if}\ \mathcal{X}=\varnothing;\\
&\bigl\{U\in\mathcal{F}\mid U\subseteq\bigcup\mathcal{X}\bigr\},\quad\text{if}\ \mathcal{X}\ne\varnothing.
\end{cases}
\]
It is straightforward to check that $\varphi_\mathcal{F}$ is an algebraic closure operator whenever $\mathcal{F}\subseteq\mathcal{T}(\mathbb{X})$
consists of compact open sets and that
$\langle\mathcal{F};\subseteq,\varphi_\mathcal{F}\rangle$
is a distributive $c$-poset in this case.

\begin{lemma}\label{L:ds}
Let $\mathcal{P}_0=\langle P_0;\leq,\varphi_0\rangle$ and $\mathcal{P}_0=\langle P_1;\leq,\varphi_1\rangle$
be distributive $c$-posets and let $\xi\colon\mathcal{P}_0\to\mathcal{P}_1$ be a $\mathbf{DP}$-isomorphism.
If $\varphi_1(A)=\dnw A$ for all $A\subseteq P_1$ then $\varphi_0(A)=\dnw A$ for all $A\subseteq P_0$.
\end{lemma}

\begin{proof}
It follows by \cite[Lemma 1.8(i)]{S1} that $\xi\colon\langle P_0;\leq\rangle\to\langle P_1;\leq\rangle$ is an isomorphism of posets.

Let $A\subseteq P_0$, let $B=\xi\varphi_0(A)$, and let $c\leq b$ for some $b\in B$.
Then $\xi^{-1}(c)\leq\xi^{-1}(b)\in\varphi_0(A)$ whence $\xi^{-1}(c)\in\varphi_0(A)$ as $\varphi_0(A)$ is a $\varphi_0$-ideal.
This implies that $c=\xi\xi^{-1}(c)\in\xi\varphi_0(A)=B$ and $B=\dnw B=\varphi_1(B)$.
As $A\subseteq\varphi_0(A)$, we conclude that $\xi(A)\subseteq B$ whence $\varphi_1\xi(A)\subseteq\varphi_1(B)=B$.
Conversely, suppose that there exists $b\in B$ such that $b\notin\varphi_1\xi(A)$. Since $\varphi_1\xi(A)$ is a $\varphi_1$-ideal,
we conclude by Theorem \ref{BS} that there is a prime $\varphi_1$-ideal $I$ such that $\xi(A)\subseteq\varphi_1\xi(A)\subseteq I$ and $b\notin I$.
Thus, $A\subseteq\xi^{-1}(I)$. Moreover, $\xi^{-1}(I)$ is a prime $\varphi_0$-ideal as $\xi$ is a $\mathbf{DP}$-morphism.
Hence, $\varphi_0(A)\subseteq\xi^{-1}(I)$ and $b\in B=\xi\varphi_0(A)\subseteq\xi\xi^{-1}(I)=I$ which is a contradiction.
This contradiction proves the inclusion $B\subseteq\varphi_1\xi(A)$ and therefore the equality $\varphi_1\xi(A)=\xi\varphi_0(A)$.

Now, let $b\in\varphi_0(A)$; then $\xi(b)\in\xi\varphi_0(A)=\varphi_1\xi(A)$. By our assumption, there is $a\in A$ such that $\xi(b)\leq\xi(a)$.
Since $\xi$ is an isomorphism of posets, we conclude that $b\leq a\in A$ and $\varphi_0(A)\subseteq\dnw A\subseteq\varphi_0(A)$.
Thus, $\varphi_0(A)=\dnw A$ for all $A\subseteq P_0$.
\end{proof}

\noindent
For an almost sober space with base $\mathbb{X}$, we put
$\mathfrak{B}_\mathbb{X}=\langle\mathcal{B}(\mathbb{X});\subseteq,\varphi_{\mathcal{B}(\mathbb{X})}\rangle$.

\begin{theorem}\label{T:TP}
\cite[Theorem 3.4]{S1}
Let $\mathbb{X}=\langle X,\mathcal{T}(\mathbb{X}),\mathcal{B}\rangle$ be an almost sober space.
Then the mapping
\[
f_\mathbb{X}\colon\mathbb{X}\to\operatorname{\mathbb{S}pec}\mathfrak{B}_\mathbb{X},\quad
f_\mathbb{X}\colon x\mapsto\{V\in\mathcal{B}\mid x\notin V\}
\]
is a homeomorphism. Moreover, $f_\mathbb{X}^{-1}(V_A)=A$ for all $A\in\mathcal{B}$
whence $f_\mathbb{X}$ is a homeomorphism of topological spaces with base.
\end{theorem}

\medskip

\subsection{Functors $\mathsf{P}$ and $\mathsf{T}$}
{\ }

We consider the following two functors.
\begin{align*}
&\mathsf{P}\colon\mathbf{AS}\to\mathbf{DP};\\
&\mathsf{P}\colon\mathbb{X}\mapsto\langle\mathcal{B}(\mathbb{X});\subseteq,\varphi_\mathcal{B}\rangle;\\
&\text{if}\ f\colon\mathbb{X}_0\to\mathbb{X}_1\ \text{is spectral then}\ 
\mathsf{P}(f)\colon U\mapsto f^{-1}(U).
\end{align*}

\begin{align*}
&\mathsf{T}\colon\mathbf{DP}\to\mathbf{AS};\\
&\mathsf{T}\colon\mathcal{P}\mapsto\operatorname{\mathbb{S}pec}\mathcal{P};\\
&\text{if}\ f\colon\mathcal{P}_0\to\mathcal{P}_1\ \text{is a}\ \mathbf{DP}\text{-morphism then}\ 
\mathsf{T}(f)\colon I\mapsto f^{-1}(I).
\end{align*}

\begin{theorem}\label{T:Stone}
\cite[Theorem 5.5]{S1}
Functors $\mathsf{P}$ and $\mathsf{T}$ establish the dual equivalence of the categories $\mathbf{AS}$ and $\mathbf{DP}$.
\end{theorem}

\medskip

\section{Computable versions of dualities I: General objects}\label{sect:gen-obj}

\noindent
Here we discuss how the duality of Theorem~\ref{T:Stone} could be made effective. In Subsection~\ref{subsect:pres-obj}, we give our `effectivized' notions of $c$-posets and spaces with base (Definitions~\ref{def:comp-c-poset} and~\ref{def:comp-top-space-with-base}). Theorem~\ref{prop:complexity-presentations} and Corollary~\ref{C:Stone-Z} show that in a sense, these effectivized notions are perfectly aligned: one can establish the dual equivalence of the corresponding full subcategories. We postpone the discussion of the effective morphisms until Subsection~\ref{subsect:morphisms}.

In Subsection~\ref{subsect:deg-spectra}, we give the first application of our framework. We transfer the well-known notion of the \emph{degree spectrum} of a countable algebraic structure (not to be confused with the spectra of $c$-posets from Definition~\ref{D:SpL}) into our setting (see Definitions~\ref{def:Z-comp-things} and~\ref{def:spectra-of-spaces}). Theorem~\ref{theo:degree-spectra} proves that every such degree spectrum of a structure can also be realized as the spectrum of a second-countable space with base.

First, we give the necessary background on computability theory.
Here we view $\mathcal{P}(\omega)$ as a topological space equipped with the \emph{Scott topology}. The basis for the Scott topology contains sets
\[
	[D] = \{ A \subseteq \omega\mid D \subseteq A\}, \text{ for } D\subseteq_{\mathrm{fin}} \omega.
\]

We fix the standard numbering of the family of all finite subsets of $\omega$: $D_0 = \varnothing$, and if $k = 2^{x_0} + 2^{x_1} + \dots + 2^{x_m}$ for $x_0 < x_1 < \dots < x_m$, then $D_k = \{ x_0, x_1, \dots, x_m\}$. As usual, for $x,y\in\omega$ 
\[
	\langle x,y\rangle = \frac{(x+y)(x+y+1)}{2} + x.
\]

For a set $Z\subseteq \omega$ we will say below that a function $f$ is \emph{$Z$-computable} if $f$ can be computed on a Turing machine with the oracle $Z$. Informally this means that the function $f$ can be computed using membership queries for $Z$. A set $X\subseteq\omega$ is  \emph{$Z$-computable} if the characteristic function of $X$ is $Z$-computable. 
A set $X\subseteq\omega$ is  \emph{$Z$-computably enumerable} (\emph{$Z$-c.e.}) if $X$ is either empty or the range of a $Z$-computable function. Informally this means that  $X$ can be algorithmically enumerated using membership queries for $Z$. These notions are standard for computability theory, see, e.g., \cite{Soare} and \cite{DowneyMelnikov}. Also  we will use below  the standard list of all $\emptyset$-computably enumerable (c.e.) sets $(W_i)_{i\in\omega}$.

\medskip

\begin{definition}\cite[Definition~3.18]{Case}
	A set $A\subseteq \omega$ defines a \emph{generalized enumeration operator} $\Gamma_A\colon \mathcal{P}(\omega) \to \mathcal{P}(\omega)$ iff for every set $B \subseteq \omega$, we have
	\[
		\Gamma_A (B) = \bigl\{ x\mid\exists k ( \langle x, k\rangle \in A \ \&\ D_k \subseteq B )\bigr\}.
	\]
\end{definition}

\noindent
The following result is well-known.

\begin{proposition}[folklore]
	A function $\Gamma \colon \mathcal{P}(\omega) \to \mathcal{P}(\omega)$ is continuous if and only if $\Gamma$ is a generalized enumeration operator. 
\end{proposition}

\noindent
We use notation $\Gamma_i = \Gamma_{W_{i}}$, and we say that $\Gamma_i$ is an \emph{enumeration operator}. It is clear that $\Gamma_i(C)$ is c.e.\ if $C$ is c.e.

Suppose that $B,C \subseteq \omega$. The set $B$ is \emph{enumeration reducible} to $C$ (denoted by $B \leq_e C$) if there exists an enumeration operator $\Gamma_i$ such that $B = \Gamma_i (C)$.

\bigskip

\subsection{Presentations of objects}\label{subsect:pres-obj}
{\ }

Let $Z \subseteq \omega$. We introduce the following notions.

\begin{definition}\label{def:comp-c-poset}
	We say that a $c$-poset $\mathcal{P}=\langle P;\leq,\varphi\rangle$ is \emph{$Z$-computably enumerable} if $\mathcal{P}$ has the following properties:
	\begin{itemize}
		\item $\langle P;\leq\rangle$ is a $Z$-c.e.\ poset in the standard sense of computable structure theory (that is, $P$ is a $Z$-computable subset of $\omega$, and the relation $\leq$ is $Z$-computably enumerable);
		
		\item $\varphi$ equals to a generalized enumeration operator $\Gamma_A$ such that $A$ is $Z$-c.e.
	\end{itemize}
\end{definition}

\medskip

\begin{definition}\label{def:comp-top-space-with-base}
	Suppose that $\mathbb{X}=\langle X,\mathcal{T},\mathcal{B}\rangle$ is a topological space with base such that its basis $\mathcal{B}$ is at most countable. Let $\beta$ be a surjection acting from $\omega$ onto $\mathcal{B}$. We say that $\langle X,\mathcal{T},\mathcal{B},\beta\rangle$ is a \emph{$Z$-computably enumerable space with base} if:
	\begin{itemize}
		\item the set $\{ (i,j)\mid \beta(i) \neq \beta(j)\}$ is $Z$-computably enumerable, and 
		\item the binary predicate
		\begin{equation}\label{equ:ce-space}
			\mathrm{Inc}_{\mathcal{B}} = \bigl\{ (i,k)\mid\beta(i) \subseteq \bigcup_{j\in D_k} \beta(j),\ D_k \neq \varnothing \bigr\}
		\end{equation}
		is also $Z$-c.e.
	\end{itemize}
\end{definition}

\noindent
We note that Definition~\ref{def:comp-top-space-with-base} has a flavor that is similar to the familiar notion of a \emph{computable topological space} (\cite[Definition~3.1]{GW}, see also \cite[Definition~1.3]{KK-17}).

We will sometimes write $\langle X,\mathcal{T}, (B_i)_{i\in\omega} \rangle$ in place of $\langle X,\mathcal{T},\mathcal{B},\beta\rangle$ meaning that $\beta(i) = B_i$ for $i\in\omega$ and $\mathcal{B} = \{ B_i\mid i\in\omega\}$. We establish the following standard result:

\begin{lemma}\label{lem:injective}
For an oracle $Z\subseteq \omega$, the following statements hold.
\begin{enumerate}
\item[{\rm (a)}]
Let $\mathcal{P}=\langle P;\leq,\varphi\rangle$ be a $Z$-computably enumerable $c$-poset. If the set $P$ is $($countably$)$ infinite, then $\mathcal{P}$ is isomorphic to a $Z$-computably enumerable $c$-poset of the form $\tilde{\mathcal{P}} = \langle \omega; \leq_{\tilde{\mathcal{P}}}, \psi\rangle$.
\item[{\rm (b)}]
Let $\langle X,\mathcal{T},\mathcal{B}, \beta\rangle$ be a $Z$-computably enumerable space with base such that the basis $\mathcal{B}$ is $($countably$)$ infinite. Then there exists a bijection $\tilde\beta \colon \omega \to \mathcal{B}$ such that the space $\langle X,\mathcal{T},\mathcal{B}, \tilde\beta\rangle$ is also $Z$-computably enumerable.
\end{enumerate}
\end{lemma}

\begin{proof}
(a)\ Since the set $P$ is $Z$-computable and infinite, one can choose a $Z$-computable bijection $f\colon \omega\to P$. We put:
	\begin{itemize}
		\item[(i)]
		$i \leq_{\tilde{\mathcal{P}}} j$ if and only if $f(i) \leq f(j)$,
		
		\item[(ii)]
		for $X\subseteq \omega$, $\psi(X) = f^{-1}(\varphi(f(X)))$.
	\end{itemize}
	If $\varphi = \Gamma_A$, then one can deduce that $\psi = \Gamma_{B}$, where
	\[
		B = \bigl\{ \langle f^{-1}(x), k^\prime\rangle\mid \langle x,k\rangle \in A \text{ and } D_k = f(D_{k^\prime})\bigr\}.
	\]
	Since $A$ is $Z$-c.e., the set $B$ is also $Z$-c.e. Hence, the $c$-poset $\tilde{\mathcal{P}} = \langle \omega; \leq_{\tilde{\mathcal{P}}}, \psi\rangle$ is $Z$-c.e.
	In view of (i)--(ii) above, the map $f$ is a $\mathbf{DP}$-isomorphism.

	(b)\ Definition~\ref{def:comp-top-space-with-base} implies that the set $V = \bigl\{i<\omega\mid\forall j < i\ \ \beta(j) \neq \beta(i)\bigr\}$ is an infinite $Z$-computable set. Thus, one can choose a $Z$-computable bijection $f$ acting from $\omega$ onto $V$. We put $\tilde\beta = \beta \circ f$.
\end{proof}

\medskip
\noindent
We establish the following result on the complexity of presentations for (countably presentable) objects in the categories $\mathbf{AS}$ and $\mathbf{DP}$.

\begin{theorem}\label{prop:complexity-presentations}
	For an oracle $Z\subseteq \omega$ and an almost semispectral space with base $\mathbb{X}=\langle X,\mathcal{T},\mathcal{B}\rangle$, the following conditions are equivalent:
	\begin{itemize}
		\item[(a)] the space $\mathbb{X}$ is $\mathbf{AS}$-isomorphic to a $Z$-computably enumerable topological space with base,
		
		\item[(b)] the dual $c$-poset $\langle\mathcal{B}(\mathbb{X});\subseteq,\varphi_\mathcal{B}\rangle$  is $\mathbf{DP}$-isomorphic to a $Z$-computably enumerable $c$-poset.
	\end{itemize}
\end{theorem}

\begin{proof}
	Without loss of generality, we may assume that the basis $\mathcal{B}$ is countably infinite (the case of a finite basis $\mathcal{B}$ can be treated in a similar way). 

	\underline{(a)$\Longrightarrow$(b).} Let $\langle X,\mathcal{T}, (B_i)_{i\in\omega} \rangle$ be a $Z$-computably enumerable space with base. By Lemma~\ref{lem:injective}.(b), we may assume that $B_i \neq B_j$ for $i\neq j$.
	
	Consider the dual $c$-poset $\mathcal{P} = \bigl\langle \{B_i\mid i\in\omega\}; \subseteq, \varphi_{\mathcal{B}}\bigr\rangle$,
	where for $\mathcal{X} \subseteq \{ B_i\mid i\in\omega\}$, we have
	\[
		\varphi_{\mathcal{B}}(\mathcal{X}) = \begin{cases}
			\varnothing, & \text{if } \mathcal{X} = \varnothing,\\
			\bigl\{ B_j\mid B_j \subseteq \bigcup \mathcal{X}\bigr\}, & \text{if } \mathcal{X} \neq \varnothing.
		\end{cases}
	\]
	If we identify a basic set $B_i$ with its index $i$, then we may assume that the $c$-poset $\mathcal{P}$ has form $\langle \omega; \leq_{\mathcal{P}}, \varphi\rangle$, where:
	\begin{itemize}
		\item $i\leq_{\mathcal{P}} j$ if and only if $B_i \subseteq B_j$ (this is a $Z$-c.e.\ property, since the binary predicate $\mathrm{Inc}_{\mathcal{B}}$ from Eq.~(\ref{equ:ce-space}) is $Z$-c.e.);
		
		\item for $X\subseteq \omega$, we have
		\[
			\varphi(X) = \begin{cases}
			\varnothing, & \text{if } X = \varnothing,\\
			\bigl\{ j\mid B_j \subseteq \bigcup_{\ell \in D_k} B_{\ell} \text{ for some } \varnothing \neq D_k \subseteq X\bigr\}, & \text{if } X \neq \varnothing.
		\end{cases}
		\]
	\end{itemize}	
	
	Since $\mathrm{Inc}_{\mathcal{B}}$ is $Z$-c.e., the set 
	\begin{equation*}
		A = \Bigl\{ \langle j,k\rangle\mid B_j \subseteq \bigcup_{\ell\in D_k} B_\ell,\ D_k\neq \varnothing \Bigr\}
	\end{equation*}
	is also $Z$-c.e. Furthermore, observe that $\varphi$ is equal to the generalized enumeration operator $\Gamma_A$. Since the set $A$ is $Z$-c.e., we deduce that the $c$-poset $\mathcal{P}$ is $Z$-c.e.

	\smallskip
		
	\underline{(b)$\Longrightarrow$(a).} By Lemma~\ref{lem:injective}.(b), we may assume that the $c$-poset $\langle \mathcal{B}; \subseteq, \varphi_{\mathcal{B}}\rangle$ is $\mathbf{DP}$-isomorphic to a $Z$-c.e.\ $c$-poset $\mathcal{P} = \langle \omega; \leq, \varphi\rangle$, where $\varphi = \Gamma_C$ is a generalized enumeration operator.
		
	Consider the dual space $\operatorname{\mathbb{S}pec}\mathcal{P} = \langle\operatorname{Spec}\mathcal{P},\mathcal{T}_{\mathcal{P}}, \mathcal{B}_{\mathcal{P}}\rangle$, where $\mathcal{B}_{\mathcal{P}} = \{ V_a\mid a\in \omega\}$ 
	and
	\[
		V_a = \{ I \in \operatorname{Spec}\mathcal{P}\mid a\not\in I\}.
	\]
	For $a\in\omega$, we put $\beta(a) = V_a$. Then we have $\beta(a) \neq \beta(b)$ for all $a\neq b$
	by Lemma \ref{L:phi}(iii).
	
	For a finite set $F\subseteq_{\mathrm{fin}}\omega$, we have $V_a \subseteq \bigcup_{b\in F} V_b= \bigcup_{c\in \varphi(F)} V_c$ if and only if $a\in \varphi(F)$
	by Lemma \ref{L:phi}(ii). Therefore, the set
	\begin{gather*}
		\mathrm{Inc}_{\mathcal{B}_{\mathcal{P}}} = \Bigl\{ (i,k)\mid V_i \subseteq \bigcup_{j\in D_k} V_j,\ D_k \neq \varnothing \Bigr\} =
		\bigl\{ (i,k)\mid i \in \varphi(D_k),\ D_k \neq \varnothing\bigr\} =\\
		 \bigl\{ (i,k)\mid \langle i,k^\prime\rangle\in C \text{ for some } \varnothing\neq D_{k^\prime} \subseteq D_k\bigr\}
	\end{gather*}
	is $Z$-c.e. Therefore, the space $\operatorname{\mathbb{S}pec}(\mathcal{P})$ is $Z$-c.e. Proposition~\ref{prop:complexity-presentations} is proved.
\end{proof}

\noindent
Let $\mathbf{AS}^Z$ denote the full subcategory of $\mathbf{AS}$ whose objects are the almost semispectral spaces with base
which are $\mathbf{AS}$-isomorphic to $Z$-computably enumerable
almost semispectral topological spaces with base. Let also
$\mathbf{DP}^Z$ denote the full subcategory of $\mathbf{DP}$ whose objects are the $c$-distributive posets
which are $\mathbf{DP}$-isomorphic to $Z$-computably enumerable
$c$-distributive posets.

From Theorems \ref{T:Stone} and \ref{prop:complexity-presentations}, we obtain the following

\begin{corollary}\label{C:Stone-Z}
For an oracle $Z\subseteq \omega$, the two categories, $\mathbf{AS}^Z$ and $\mathbf{DP}^Z$, are dually equivalent.
\end{corollary}

\medskip

\subsection{On degree spectra of structures}\label{subsect:deg-spectra}
{\ }

In this subsection, we consider only computable signatures $\sigma$. A $\sigma$-formula $\psi(x_1,x_2,\dots,x_k)$ is identified with its G{\"o}del number.

If $\mathcal{A}$ is a $\sigma$-structure with $\operatorname{dom}(\mathcal{A}) \subseteq \omega$, then by $D(\mathcal{A})$ we denote the atomic diagram of the structure $\mathcal{A}$.

\smallskip

Let $\mathcal{S}$ be a countably infinite $\sigma$-structure. The \emph{degree spectrum} of the structure $\mathcal{S}$ is the following set of Turing degrees:
\[
	\operatorname{DgSp}(\mathcal{S}) = \{ \deg_T(D(\mathcal{A}))\mid \mathcal{A} \cong \mathcal{S},\ \operatorname{dom}(\mathcal{A}) = \omega\}.
\]

A structure $\mathcal{S}$ is \emph{automorphically trivial} if there exists a finite set $F\subseteq \operatorname{dom}(\mathcal{S})$ such that every permutation of $\operatorname{dom}(\mathcal{S})$ that keeps the set $F$ fixed is an automorphism of the structure $\mathcal{S}$. Knight~\cite{Knight-86} proved that, for every automorphically nontrivial, countable structure $\mathcal{S}$, its degree spectrum $\operatorname{DgSp}(\mathcal{S})$ is closed upwards in Turing degrees. Hence, for such a structure $\mathcal{S}$ we have
\begin{multline}\label{equ:DgSp}
	\operatorname{DgSp}(\mathcal{S}) = \{ \mathbf{d}\mid \mathbf{d} \text{ is a Turing degree such that}\\ 
	\mathcal{S} \text{ has a } \mathbf{d}\text{-computable isomorphic copy} \}.
\end{multline}
The following result is well-known.

\begin{proposition}[see, e.g., Theorem~3.2 of~\cite{HKSS}]\label{prop:HKSS}
	For every automorphically non-trivial, countable $\sigma$-structure $\mathcal{S}$, there exists a countable poset $\mathcal{P}_{\mathcal{S}} = (\omega;\leq_{\mathcal{S}})$ such that $\operatorname{DgSp}(\mathcal{P}_\mathcal{S}) = \operatorname{DgSp}(\mathcal{S})$.
\end{proposition}

\medskip
\noindent
Motivated by the degree spectra of structures, we introduce the following notions:

\begin{definition}\label{def:Z-comp-things}
We say that a $c$-poset $\mathcal{P}=\langle P;\leq,\varphi\rangle$ is \emph{$Z$-computable} if $\mathcal{P}$ has the following properties:
	\begin{itemize}
		\item $\langle P;\leq\rangle$ is a $Z$-computable poset (i.e., $P$ is a $Z$-computable subset of $\omega$, and the relation $\leq$ is $Z$-computable);
		
		\item $\varphi$ equals to a generalized enumeration operator $\Gamma_A$ satisfying $A\leq_T Z$.
	\end{itemize}

Let $\mathbb{X}=\langle X,\mathcal{T},\mathcal{B}\rangle$ be a topological space with base, and let $\beta$ be a surjection from $\omega$ onto $\mathcal{B}$. We say that $\langle X,\mathcal{T},\mathcal{B},\beta\rangle$ is a \emph{$Z$-computable space with base} if the binary predicate $\mathrm{Inc}_{\mathcal{B}}$ from Eq.~(\ref{equ:ce-space}) is $Z$-computable.
\end{definition}

\begin{definition}\label{def:spectra-of-spaces}
	(a)\ Let $\mathcal{P}$ be a countable $c$-poset from $\mathbf{DP}$. The \emph{degree spectrum of the $c$-poset} $\mathcal{P}$ (denoted by $\operatorname{DgSp}(\mathcal{P})$) is the set of all Turing degrees $\mathbf{d}$ such that $\mathcal{P}$ is $\mathbf{DP}$-isomorphic to a $\mathbf{d}$-computable $c$-poset.
	
	(b)\ Let $\mathbb{X}$ be a second-countable space from $\mathbf{AS}$. The \emph{degree spectrum of the second-countable topological space with base} $\mathbb{X}$ (denoted by $\operatorname{DgSp}(\mathbb{X})$) is the set of all Turing degrees $\mathbf{d}$ such that $\mathbb{X}$ is $\mathbf{AS}$-isomorphic to a $\mathbf{d}$-computable space with base.
\end{definition}

\noindent
Essentially by repeating the proof of Proposition~\ref{prop:complexity-presentations}, we obtain the following

\begin{proposition}\label{prop:deg-sp}
	Let $\mathbb{X}$ be a second-countable, almost semispectral space with base. Then its dual $c$-poset $\mathsf{P}(\mathbb{X})$ satisfies:
	\[
		\operatorname{DgSp}(\mathsf{P}(\mathbb{X})) = \operatorname{DgSp}(\mathbb{X}).
	\]
\end{proposition}

\noindent
We apply Proposition~\ref{prop:deg-sp} to obtain the following result.

\begin{theorem}\label{theo:degree-spectra}
	Suppose that $\mathcal{S}$ is an automorphically non-trivial, countable $\sigma$-structure. Then there exists a second-countable, almost semispectral space with base $\mathbb{X}_{\mathcal{S}}$ satisfying
	\[
		\operatorname{DgSp}(\mathbb{X}_{\mathcal{S}}) = \operatorname{DgSp}(\mathcal{S}).
	\]
\end{theorem}

\begin{proof}
	By Proposition~\ref{prop:HKSS}, we may assume that $\mathcal{S} = \langle \omega; \leq_{\mathcal{S}}\rangle$ is a poset. As in Example~1 of the paper~\cite{S1}, we define the closure operator
	\[
		\varphi(X) = \ \downarrow\! X = \bigl\{ a\mid (\exists b\in X) (a\leq_{\mathcal{S}} b)\bigr\}.
	\]
	Then $\mathcal{P}_{\mathcal{S}} = \langle \omega; \leq_{\mathcal{S}}, \varphi\rangle$ is a distributive $c$-poset. Notice that $\varphi$ equals to the generalized enumeration operator $\Gamma_A$, where
	\[
		A = \bigl\{ \langle a, k\rangle\mid D_k = \{b\},\ a\leq_{\mathcal{S}} b\bigr\}.
	\]
	Here $A$ is computable with respect to the oracle $\leq_{\mathcal{S}}$. By applying this observation, it is not hard to deduce that 
	\[
		\operatorname{DgSp}(\mathcal{P}_{\mathcal{S}}) = \operatorname{DgSp}(\mathcal{S}).
	\]
	We choose $\mathbb{X}_{\mathcal{S}}$ as the dual space $\mathsf{T}(\mathcal{P}_{\mathcal{S}})$. By Proposition~\ref{prop:deg-sp}, we deduce that $\operatorname{DgSp}(\mathbb{X}_{\mathcal{S}}) = \operatorname{DgSp}(\mathcal{S})$.
\end{proof}

\medskip
We note that in the recent years, degree spectra for Polish spaces \emph{up to homeomorphism} have been extensively studied~--- see, e.g., the papers~\cite{HMN-20,HKS-ta,MN-23,Mel-21,DM-23}. To our best knowledge, it is still unknown whether one can establish an analogue of Theorem~\ref{theo:degree-spectra} for this kind of degree spectra.

\medskip

\section{Computable versions of dualities II: Subcategories and morphisms}\label{sect:semilat}

\noindent
In this section, following~\cite{S1,S2}, we focus on some familiar subcategories of the category $\mathbf{DP}$. 
Subsection~\ref{subsect:lat} discusses semilattices and lattices. In computable structure theory, there is a well-established notion of, say, a computable join-semilattice or a computable lattice, so we give appropriate effective notions of topological spaces with base taken from the corresponding subcategories of $\mathbf{AS}$ (see Definition~\ref{def:comp-space-for-lat}). Similarly to Section~\ref{sect:gen-obj}, we establish results on duality and degree spectra (Theorem~\ref{theo:semi_lattices} and Corollary~\ref{corol:lattices-spectra}). 

In Subsection~\ref{subsect:morphisms}, we discuss the complexity of morphisms. We introduce the notion of effective spectral map (Definition~\ref{def:Z-comp-spectral}). Using the known results from computable structure theory, for a natural number $N\geq 1$, we build a computable topological space with base which has precisely $N$-many computable copies, up to effective spectral homeomorphisms (Corollary~\ref{cor:comp-dim}).

\subsection{Semilattices and lattices}\label{subsect:lat}
{\ }

Let $Z\subseteq \omega$, and let $f$ be a binary operation. As customary in computable structure theory, a structure $\langle P; f\rangle$ is called \emph{$Z$-computable} if $P$ is a $Z$-computable subset of $\omega$ and the function $f$ is $Z$-computable. 

The notion introduced above allows us to talk about $Z$-computable join-semilattices $\langle P; \vee\rangle$ and $Z$-computable meet-semilattices $\langle P; \wedge\rangle$. A lattice $\langle P; \vee,\wedge \rangle$ is $Z$-computable if both its semilattice reducts $\langle P; \vee\rangle$ and $\langle P; \wedge\rangle$ are $Z$-computable.

Similarly to Lemma~\ref{lem:injective}, one can prove the following result:

\begin{lemma}[folklore]\label{lem:folk-lattices}
	If $\langle P; \vee_{\mathcal{P}},\wedge_{\mathcal{P}}\rangle$ is a $Z$-computable structure $($which is not necessarily a lattice$)$ and $P$ is an infinite set, then $\langle P; \vee_{\mathcal{P}},\wedge_{\mathcal{P}}\rangle$ is isomorphic to a $Z$-computable structure of the form $\langle \omega; \vee,\wedge\rangle$. 
\end{lemma}

\medskip

We recall the following full subcategories of $\mathbf{DP}$ and their dual full subcategories of $\mathbf{AS}$ (see further details in Appendix).
\begin{itemize}
	\item $\mathbf{DSL}^{\wedge}$ (distributive meet-semilattices)\ $\mapsto$\ $\mathbf{ASp}$ (almost spectral spaces with base),
	
	\item $\mathbf{DSL}^{\vee}$ (distributive join-semilattices)\ $\mapsto$\  $\mathbf{AsSpec}$ (almost semispectral spaces with additive base),
	
	\item $\mathbf{DL}$ (distributive lattices)\ $\mapsto$\  $\mathbf{ASpec}$ (almost spectral spaces).
\end{itemize}

For objects from the full subcategories $\mathbf{ASp}$, $\mathbf{AsSpec}$, and $\mathbf{ASpec}$, we introduce the following notion of effective presentability:

\begin{definition}\label{def:comp-space-for-lat}
	Suppose that $\mathbb{X}=\langle X,\mathcal{T},\mathcal{B}\rangle$ is a topological space with base such that the base $\mathcal{B}$ is  countable. Let $\beta$ be a bijection from $\omega$ onto $\mathcal{B}$. 
	\begin{itemize}
		\item[(a)] Suppose that the space $\mathbb{X}$ is an object from $\mathbf{ASp}$. We say that $\langle X,\mathcal{T},\mathcal{B},\beta\rangle$ is a \emph{$Z$-computable $\mathbf{ASp}$-space} if there exists a computable function $f_{\wedge}(x,y)$ such that all $i,j\in\omega$ satisfy $\beta(i) \cap \beta(j) = \beta(f_{\wedge}(i,j))$.
		
		\item[(b)] Suppose that the space $\mathbb{X}$ is an object from $\mathbf{AsSpec}$. We say that $\langle X,\mathcal{T},\mathcal{B},\beta\rangle$ is a \emph{$Z$-computable $\mathbf{AsSpec}$-space} if there exists a computable function $f_{\vee}(x,y)$ such that all $i,j\in\omega$ satisfy $\beta(i) \cup \beta(j) = \beta(f_{\vee}(i,j))$.
		
		\item[(c)] Suppose that the space $\mathbb{X}$ is an object from $\mathbf{ASpec}$. We say that $\langle X,\mathcal{T},\mathcal{B},\beta\rangle$ is a \emph{$Z$-computable $\mathbf{ASpec}$-space} if $\langle X,\mathcal{T},\mathcal{B},\beta\rangle$ is both a $Z$-computable $\mathbf{ASp}$-space and a $Z$-computable $\mathbf{AsSpec}$-space.
	\end{itemize}
\end{definition}

\noindent 
In what follows, sometimes we will write $B_i$ in place of $\beta(i)$.

\medskip 

We give a computable version of dualities (on objects):

\begin{theorem}\label{theo:semi_lattices}
	For an oracle $Z\subseteq \omega$ and an $\mathbf{ASp}$-space $\mathbb{X}=\langle X,\mathcal{T},\mathcal{B}\rangle$, the following conditions are equivalent:
	\begin{itemize}
		\item[(a)] the space $\mathbb{X}$ is $\mathbf{AS}$-isomorphic to a $Z$-computable $\mathbf{ASp}$-space,
		
		\item[(b)] the dual meet-semilattice $\langle\mathcal{B}(\mathbb{X});\cap\rangle$ is $\mathbf{DP}$-isomorphic to a $Z$-computable meet-semilattice.
	\end{itemize}
	Similar dualities are true for:
	\begin{itemize}
		\item[(i)] $Z$-computable $\mathbf{AsSpec}$-spaces and $Z$-computable join semilattices,
		
		\item[(ii)] $Z$-computable $\mathbf{ASpec}$-spaces and $Z$-computable lattices.
	\end{itemize}
\end{theorem}
\begin{proof}
	We give a detailed proof only for $\mathbf{ASp}$-spaces and meet-semilattices. After the main proof, we provide a brief comment on how to establish the other two dualities.
	
	\underline{(a)$\Longrightarrow$(b).} Let $\langle X,\mathcal{T}, (B_i)_{i\in\omega} \rangle$ be a $Z$-computable $\mathbf{ASp}$-space. Consider the dual meet-semilattice $\mathcal{M} = \langle \{ B_i: i\in\omega\}; \cap\rangle$. If we identify a basic set $B_i$ with its index $i$, then we may assume that $\mathcal{M}$ has form $\langle \omega; f_{\wedge}\rangle$. Hence, $\mathcal{M}$ is isomorphic to the $Z$-computable meet-semilattice $\langle \omega; f_{\wedge}\rangle$.
	
	\underline{(b)$\Longrightarrow$(a).} By Lemma~\ref{lem:folk-lattices}, we may assume that the dual meet-semilattice $\langle\mathcal{B}(\mathbb{X});\cap\rangle$ is isomorphic to a $Z$-computable meet-semilattice $\mathcal{P} = \langle \omega; \wedge\rangle$.
	
	We consider the dual space $\operatorname{\mathbb{S}pec}\mathcal{P} = \langle\operatorname{Spec}\mathcal{P},\mathcal{T}_{\mathcal{P}}, \mathcal{B}_{\mathcal{P}}\rangle$, where $\mathcal{B}_{\mathcal{P}} = \{ V_a\mid a\in \omega\}$ 
	and
	\[
		V_a = \{ I \in \operatorname{Spec}\mathcal{P}\mid a\not\in I\}.
	\]
	For $a\in\omega$, we put $\beta(a) = V_a$. Recall that by Lemma \ref{L:phi}(iii), we have $\beta(a) \neq \beta(b)$ for all $a\neq b$.
	
	By Lemma~\ref{L:phi}(iv) we have 
	\[
		V_a \cap V_b = V_c \ \Longleftrightarrow\ a\wedge b = c.
	\]
	Thus, the $Z$-computable function $\wedge$ witnesses that the space $\operatorname{\mathbb{S}pec}\mathcal{P}$ is isomorphic to a $Z$-computable $\mathbf{ASp}$-space.
	
	\smallskip
	
	Now we give comments on the further two effective dualities:
	
	\textbf{(i)} The proof of Claim~(i) proceeds in a similar way to the main proof, except that in the proof of (b)$\Longrightarrow$(a) we need to apply Lemma~\ref{L:phi}(v) in place of Lemma~\ref{L:phi}(iv).
	
	\textbf{(ii)} Claim~(ii) is an immediate corollary of the main proof and Claim~(i).
\end{proof}

\medskip

Now we give some corollaries of Theorem~\ref{theo:semi_lattices}.
Similarly to Definition~\ref{def:spectra-of-spaces}, we introduce the following notion:

\begin{definition}
	Suppose that $\mathbf{S}$ is one of the following full subcategories of $\mathbf{AS}$: $\mathbf{ASp}$, $\mathbf{AsSpec}$, $\mathbf{ASpec}$. The \emph{degree spectrum} of a second-countable $\mathbf{S}$-space $\mathbb{X}=\langle X,\mathcal{T},\mathcal{B}\rangle$ is the set of all Turing degrees $\mathbf{d}$ such that $\mathbb{X}$ is isomorphic to a $\mathbf{d}$-computable $\mathbf{S}$-space.
\end{definition}

Theorem~\ref{theo:semi_lattices} implies the following:

\begin{corollary}\label{corol:lattices-spectra}
	Let $\mathcal{M}$ be an automorphically non-trivial, countable distributive meet-semilattice. Then for the second-countable $\mathbf{ASp}$-space $\operatorname{\mathbb{S}pec}\mathcal{M}$, its degree spectrum is equal to the degree spectrum of $\mathcal{M}$.
	
	A similar result is true for:
	\begin{itemize}
		\item distributive join-semilattices  $\mathcal{J}$ and $\mathbf{AsSpec}$-spaces $\operatorname{\mathbb{S}pec} \mathcal{J}$,
		
		\item distributive lattices  $\mathcal{L}$ and $\mathbf{ASpec}$-spaces $\operatorname{\mathbb{S}pec} \mathcal{L}$.
	\end{itemize}
\end{corollary}

Theorem~1.1 of~\cite{BFKM} established that there exists an automorphically non-trivial, distributive lattice $\mathcal{L}$ such that $\mathcal{L}$ has the Slaman--Wehner degree spectrum~\cite{Slaman,Wehner}, that is, $\operatorname{DgSp}(\mathcal{L}) = \{ \mathbf{d} \mid \mathbf{d} \text{ is non-computable}\}$. Thus, by applying Corollary~\ref{corol:lattices-spectra}, we obtain the following fact.

\begin{corollary}\label{corol:Slaman-Wehner}
	There exists a second-countable $\mathbf{ASpec}$-space such that its degree spectrum is equal to the set of all non-computable Turing degrees.
\end{corollary}

Recall that a Stone space is a compact and totally disconnected Hausdorff space. It is well-known that the category of Stone spaces is dually equivalent to the category of Boolean algebras.

We note that Corollary~\ref{corol:Slaman-Wehner} contrasts with the known results on degree spectra (up to homeomorphism) of Stone spaces. Indeed, the degree spectrum of a separable Stone space $\mathbb{S}$ is equal to the degree spectrum of the corresponding dual Boolean algebra $\mathcal{B}$ (\cite[Theorem~1.1]{HMN-20} and \cite[Corollary~3.4]{HKS-ta}). It is known that for a countable Boolean algebra $\mathcal{B}$, if the degree spectrum $\operatorname{DgSp}(\mathcal{B})$ contains a $\operatorname{low}_4$ Turing degree, then $\operatorname{DgSp}(\mathcal{B})$ must contain the computable degree \cite{KnightStob}. Therefore, a Stone space $\mathbb{S}$ cannot have degree spectrum equal to the set of all non-computable degrees.

\medskip

\subsection{Complexity of morphisms}\label{subsect:morphisms}
{\ }

For simplicity, in this subsection we work with computably enumerable presentations of $c$-posets and of topological spaces with bases. We note that the discussed material allows a straightforward adaptation to the case of $Y$-c.e.\ and $Y$-computable presentations, for $Y\subseteq \omega$.

\begin{definition}\label{def:Z-comp-spectral}
Let $\mathbb{X}_0 = \langle X_0, \mathcal{T}_0, (B_i)_{i\in\omega}\rangle$ and $\mathbb{X}_1 = \langle X_1, \mathcal{T}_1, (C_i)_{i\in\omega}\rangle$ be c.e.\ topological spaces with base, and let $Z\subseteq \omega$. We say that a spectral map $f\colon X_0 \to X_1$ is \emph{$Z$-effective} if there exists a $Z$-computable function $h\colon \omega \to \omega$ such that $f^{-1}(C_i) = B_{h(i)}$ for all $i\in\omega$.
\end{definition}

\noindent
We say that a $\mathbf{DP}$-morphism $g$ is $Z$-computable if $g$ is a partial $Z$-computable function, in the usual recursion-theoretic sense. (In order for this notion to be formally correct, sometimes we have to identify a basic set $B_i$ with its index $i\in\omega$.)

We establish the following result on the complexity of morphisms.

\begin{lemma}\label{lem:morphisms}
For an oracle $Z\subseteq\omega$, the following statements hold.
\begin{enumerate}
\item[{\rm (a)}]
Let $\mathbb{X}_0$ and $\mathbb{X}_1$ be c.e.\ spaces from $\mathbf{AS}$. If $f\colon \mathbb{X}_0 \to \mathbb{X}_1$ is a $Z$-effective spectral map, then $g = \mathsf{P}(f)$ is a $Z$-computable $\mathbf{DP}$-morphism from $\mathsf{P}(\mathbb{X}_1)$ to $\mathsf{P}(\mathbb{X}_0)$.
\item[{\rm (b)}]
Let $\mathcal{P}_0$ and $\mathcal{P}_1$ be c.e.\ $c$-posets from $\mathbf{DP}$. If $g\colon \mathcal{P}_0 \to \mathcal{P}_1$ is a $Z$-computable $\mathbf{DP}$-morphism, then $f = \mathsf{T}(g)$ is a $Z$-effective spectral map from $\mathsf{T}(\mathcal{P}_1)$ to $\mathsf{T}(\mathcal{P}_0)$.
\end{enumerate}
\end{lemma}

\begin{proof}
	(a)\ Suppose that $\mathbb{X}_0 = \langle X_0, \mathcal{T}_0, (B_i)_{i\in\omega}\rangle$ and $\mathbb{X}_1 = \langle X_1, \mathcal{T}_1, (C_i)_{i\in\omega}\rangle$. Let $h$ be a $Z$-computable function witnessing that $f$ is a $Z$-effective spectral map. Then observe that $g(C_i) = f^{-1}(C_i) = B_{h(i)}$, and thus, by identifying basic sets with their indices, we get that the morphism $g=h$ is $Z$-computable.
	
	(b)\ Suppose that $\mathcal{P}_0 = \langle P_0; \leq_0, \varphi_0\rangle$ and $\mathcal{P}_1 = \langle P_1; \leq_1, \varphi_1\rangle$. Let $g$ be a $Z$-computable $\mathbf{DP}$-morphism acting from $P_0$ to $P_1$. Then the spectral map $f = \mathsf{T}(g)$ satisfies the following (see the proof of Claim~5.1 in~\cite{S1}): for $a\in P_0$, $f^{-1}(V_a)  = V_{g(a)}$. Hence, the $Z$-computable function $g$ witnesses that $f$ is a $Z$-effective spectral map from  $\mathsf{T}(\mathcal{P}_1)$ into $\mathsf{T}(\mathcal{P}_0)$.
\end{proof}

\noindent
The following interesting result is a corollary of Lemma~\ref{lem:morphisms}.

\medskip
For a computable $\sigma$-structure $\mathcal{S}$, the \emph{computable dimension} of $\mathcal{S}$ is the number of computable isomorphic copies of $\mathcal{S}$ up to computable isomorphisms. S.\,S. Goncharov~\cite{Gon-80} proved the following

\begin{theorem}\label{T:SSG}
\cite{Gon-80}
For each non-zero natural number $N$, there exists a computable structure $\mathcal{A}$ having computable dimension $N$.
\end{theorem}

\begin{corollary}\label{cor:comp-dim}
	Let $N$ be a non-zero natural number. There exists a computable topological space with base $\mathbb{X}$ such that $\mathbb{X}$ has precisely $N$-many computable copies, up to effective spectral homeomorphisms.
\end{corollary}

\begin{proof}
It follows from Goncharov's Theorem \ref{T:SSG} and the results of~\cite{HKSS} that there is a computable poset $\mathcal{S}$ having computable dimension $N$.

Let $\mathcal{S}_0$,\ldots, $\mathcal{S}_{N-1}$ be the computable isomorphic copies of $\mathcal{S}$ up to computable isomorphism.
For each $i<N$, we consider the $c$-poset $\mathcal{P}_i=\mathcal{P}_{\mathcal{S}_i}=\langle S_i;\leq_i,\varphi\rangle$ and the $c$-poset
$\mathcal{P}=\mathcal{P}_\mathcal{S}=\langle S;\leq,\varphi\rangle$ where $\varphi(a)=\dnw A$ for each $i<N$ and all $A\subseteq S_i$ as well as for all $A\subseteq S$,
as defined in the proof of Theorem~\ref{theo:degree-spectra}.
Let $\mathbb{X}=\mathsf{T}(\mathcal{P})$.
For each $i<N$, let also $\mathbb{X}_i=\mathsf{T}(\mathcal{P}_i)$. We put $\mathcal{B}_i=\mathcal{B}(\mathbb{X}_i)$ and
\[
\beta_i\colon\omega\to\mathcal{B}_i;\quad
\beta_i\colon n\mapsto V_{f_i(n)},
\]
where $f_i$ is a computable function such that $S_i=f_i(\omega)$. Then we have by Lemma \ref{L:phi}(ii):
\begin{align*}
\operatorname{Inc}_{\mathcal{B}_i}=
&\Bigl\{(j,k)\in\omega^2\mid k\ne 0\ \text{and}\ \beta_i(j)\subseteq\bigcup_{m\in D_k}\beta_i(m)\Bigr\}=\\
&\Bigl\{(j,k)\in\omega^2\mid k\ne 0\ \text{and}\ f_i(j)\in\varphi\bigl(f_i(D_k)\bigr)\Bigr\}=\\
&\bigl\{(j,k)\in\omega^2\mid k\ne 0\ \text{and}\ f_i(j)\leq_i f_i(m)\ \text{for some}\ m\in D_k\bigr\}.
\end{align*}
Since $f_i$ and $\leq_i$ are computable, we conclude that $\operatorname{Inc}_{\mathcal{B}_i}$ is a computable set.
Therefore, $\mathbb{X}_i$ is a computable topological space with base for all $i<N$.

Let $\mathbb{X}^\prime=\langle X^\prime,\mathcal{T}^\prime,\mathcal{B}^\prime\rangle$ be a computable topological space with base such that
$\mathbb{X}^\prime$ and $\mathbb{X}$ are $\mathbf{AS}$-isomorphic.
By definition, there is $\beta\colon\omega\to\mathcal{B}^\prime$ which makes the set $\operatorname{Inc}_{\mathcal{B}^\prime}$ computable.
By duality, the $c$-posets
$\mathcal{P}^\prime=\mathsf{P}(\mathbb{X}^\prime)=\langle\mathcal{B}^\prime;\subseteq,\varphi_{\mathcal{B}^\prime}\rangle$ and $\mathcal{P}$
are $\mathbf{DP}$-isomorphic whence $\langle\mathcal{B}^\prime;\subseteq\rangle$ and $\mathcal{S}$ are isomorphic as posets
and $\varphi_{\mathcal{B}^\prime}(\mathcal{C})=\dnw\mathcal{C}$ for all $\mathcal{C}\subseteq\mathcal{B}^\prime$ by Lemma \ref{L:ds}.
We show that $\langle\mathcal{B}^\prime;\subseteq\rangle$ is a computable poset. Indeed, consider the following function
\[
\begin{cases}
&u(0)=0;\\
&u(n+1)=\mu\,i\,\bigl[\beta(i)\notin\bigl\{\beta(0),\ldots,\beta(n)\bigr\}\bigr].
\end{cases}
\]
Since $\mathcal{B}^\prime$ is a countable base, the function $u$ is well-defined and a strictly increasing computable function.
Hence, $\operatorname{im}u$ is a computable set.
Identifying each basic set from $\mathcal{B}^\prime$ with its number, we conclude that $\mathcal{B}^\prime$ is computable.
Since the set $\bigl\{(i,j)\in\omega^2\mid\beta(u(i))\subseteq\beta(u(j))\bigr\}$ is computable, we conclude that
$\langle\mathcal{B}^\prime;\subseteq\rangle$ is indeed a computable poset.
By our assumption, there is $i<N$ and a computable isomorphism $\alpha\colon\mathcal{S}_i\to\langle\mathcal{B}^\prime;\subseteq\rangle$.
Hence, $\alpha\colon\mathcal{P}_i\to\mathcal{P}^\prime$ is a computable $\mathbf{DP}$-isomorphism.

By Lemma~\ref{lem:morphisms}(ii) and duality,
$\mathsf{T}(\alpha)\colon\mathsf{T}(\mathcal{P}^\prime)\to\mathsf{T}(\mathcal{P}_i)=\mathbb{X}_i$ is an effective $\mathbf{AS}$-isomorphism.
It follows from Theorem \ref{T:TP} that $f_{\mathbb{X}^\prime}\colon\mathbb{X}^\prime\to\mathsf{T}\mathsf{P}(\mathbb{X}^\prime)=\mathsf{T}(\mathcal{P}^\prime)$
is also an effective $\mathbf{AS}$-isomorphism. Hence, $f_{\mathbb{X}^\prime}\mathsf{T}(\alpha)\colon\mathbb{X}^\prime\to\mathbb{X}_i$
is an effective $\mathbf{AS}$-isomorphism.

Finally, if there is an effective $\mathbf{AS}$-isomorphism $g\colon\mathbb{X}_i\to\mathbb{X}_j$ for some $i<j<N$ then, by Lemma \ref{lem:morphisms}(i) and duality,
$\mathsf{P}(g)$ is a computable $\mathbf{DP}$-isomorphism from $\mathsf{P}(\mathbb{X}_j)=\mathsf{P}\mathsf{T}(\mathcal{P}_j)$ to
$\mathsf{P}(\mathbb{X}_i)=\mathsf{P}\mathsf{T}(\mathcal{P}_i)$.
It follows from \cite[Claim 5.5]{S1} and Lemma \ref{L:phi}(iii) that $\xi_\mathcal{R}\colon\mathsf{P}\mathsf{T}(\mathcal{R})\to\mathcal{R}$ is a computable $\mathbf{DP}$-isomorphism for each $c$-poset $\mathcal{R}$.
Thus, $h=\xi_{\mathcal{P}_i}\mathsf{P}(g)\xi^{-1}_{\mathcal{P}_j}\colon\mathcal{P}_j\to\mathcal{P}_i$ is a computable $\mathbf{DP}$-isomorphism
which is a contradiction.

This contradiction implies that $\mathbb{X}_0$,\ldots, $\mathbb{X}_{N-1}$ are the computable $\mathbf{AS}$-isomorphic copies of $\mathbb{X}$
up to effective spectral homeomorphisms whence $\mathbb{X}$ is of computable dimension $N$.
\end{proof}

\noindent
We note that, to our best knowledge, it is still unknown whether an uncountable, computably presentable Polish space can have finitely many computable copies, up to \emph{effective homeomorphisms}.


\section*{Appendix}

\noindent
The following tableau of full subcategories of the category $\mathbf{DP}$ and their dual full subcategories of the category $\mathbf{AS}$ is presented in \cite{S2}.

\vskip5mm
\begin{center}
\begin{tabular}{|l|p{2.2in}|l|p{2.2in}|} \hline
$\mathbf{DP}$ 
& Objects: distributive $c$-posets 
& $\mathbf{AS}$ 
& Objects: almost semispectral spaces with base \\
 & Morphisms: strict mappings
 & 
& Morphisms: spectral mappings \\ \hline
$\mathbf{DP}_0$ 
& Objects: distributive $c$-posets with $0$
& $\mathbf{AS}_s$ 
& Objects: [sober] almost semispectral spa\-ces with $0$-base \\
 & Morphisms: strict mappings [preserving $0$]
 & 
& Morphisms: spectral mappings \\ \hline
$\mathbf{DP}_1$ 
& Objects: distributive $c$-posets with $1$
& $\mathbf{AS}_c$ 
& Objects: [compact] almost semispectral spa\-ces with $1$-base \\
 & Morphisms: strict mappings [preserving $1$]
 & 
& Morphisms: spectral mappings \\ \hline
\end{tabular}
\end{center}

\begin{center}
\begin{tabular}{|l|p{2.2in}|l|p{2.2in}|} \hline
$\mathbf{DP}_{01}$ 
& Objects: distributive $c$-posets with $0$ and $1$
& $\mathbf{S}$ 
& Objects: semispectral spaces with $(0,1)$-base \\
 & Morphisms: strict mappings [preserving $0$ and $1$]
 & 
& Morphisms: spectral mappings \\ \hline
$\mathbf{DSL}^\wedge$ 
& Objects: distributive $\wedge$-semi\-latt\-ices 
& $\mathbf{ASp}$ 
& Objects: almost spectral spaces with base \\
 & Morphisms: strict $\wedge$-semi\-latt\-ice homomorphisms
 & 
& Morphisms: spectral mappings \\ \hline
$\mathbf{DSL}^\wedge_0$ 
& Objects: distributive $\wedge$-semi\-latt\-ices with $0$
& $\mathbf{ASp}_s$ 
& Objects: sober almost spectral spa\-ces with base \\
 & Morphisms: strict $(\wedge,0)$-semi\-latt\-ice homomorphisms
 & 
& Morphisms: spectral mappings \\ \hline
$\mathbf{DSL}^\wedge_1$ 
& Objects: distributive $\wedge$-semi\-latt\-ices with $1$
& $\mathbf{ASp}_c$ 
& Objects: [compact] almost spectral spa\-ces with $1$-base \\
 & Morphisms: strict $(\wedge,1)$-semi\-latt\-ice homomorphisms
 & 
& Morphisms: spectral mappings \\ \hline
$\mathbf{DSL}^\wedge_{01}$ 
& Objects: distributive $\wedge$-semi\-latt\-ices with $0$ and $1$
& $\mathbf{Sp}$ 
& Objects: spectral spaces with $1$-base \\
 & Morphisms: strict $(\wedge,0,1)$-semi\-latt\-ice homomorphisms
 & 
& Morphisms: spectral mappings \\ \hline
$\mathbf{DSL}^\vee$ 
& Objects: distributive $\vee$-semi\-latt\-ices 
& $\mathbf{AsSpec}$ 
& Objects: almost semispectral spaces with additive base \\
 & Morphisms: strict $\vee$-semi\-latt\-ice homomorphisms
 & 
& Morphisms: spectral mappings \\ \hline
$\mathbf{DSL}^\vee_0$ 
& Objects: distributive $\vee$-semi\-latt\-ices with $0$
& $\mathbf{AsSpec}_s$ 
& Objects: [sober] almost semispectral spaces with additive $0$-base \\
 & Morphisms: strict $(\vee,0)$-semi\-latt\-ice homomorphisms
 & 
& Morphisms: spectral mappings \\ \hline
$\mathbf{DSL}^\vee_1$ 
& Objects: distributive $\vee$-semi\-latt\-ices with $1$
& $\mathbf{AsSpec}_c$ 
& Objects: compact almost semispectral spaces with additive base \\
 & Morphisms: strict $(\vee,1)$-semi\-latt\-ice homomorphisms
 & 
& Morphisms: spectral mappings \\ \hline
$\mathbf{DSL}^\vee_{01}$ 
& Objects: distributive $\vee$-semi\-latt\-ices with $0$ and $1$
& $\mathbf{sSpec}$ 
& Objects: semispectral spaces with additive $0$-base \\
 & Morphisms: strict $(\vee,0,1)$-semi\-latt\-ice homomorphisms
 & 
& Morphisms: spectral mappings \\ \hline
$\mathbf{DL}$ 
& Objects: distributive lattices 
& $\mathbf{ASpec}$ 
& Objects: almost spectral spaces \\
 & Morphisms: strict latt\-ice homomorphisms
 & 
& Morphisms: spectral mappings \\ \hline
$\mathbf{DL}_0$ 
& Objects: distributive lattices with $0$
& $\mathbf{ASpec}_s$ 
& Objects: sober almost spectral spaces \\
 & Morphisms: strict $0$-lattice homomorphisms
 & 
& Morphisms: spectral mappings \\ \hline
\end{tabular}
\end{center}

\begin{center}
\begin{tabular}{|l|p{2.2in}|l|p{2.2in}|} \hline
$\mathbf{DL}_1$ 
& Objects: distributive lattices with $1$
& $\mathbf{ASpec}_c$ 
& Objects: compact almost spectral spaces \\
 & Morphisms: strict $1$-latt\-ice homomorphisms
 & 
& Morphisms: spectral mappings \\ \hline
$\mathbf{DL}_{01}$ 
& Objects: distributive lattices with $0$ and $1$
& $\mathbf{Spec}$ 
& Objects: spectral spaces \\
 & Morphisms: $(0,1)$-lattice homomorphisms
 & 
& Morphisms: spectral mappings \\ \hline
\end{tabular}
\end{center}

\end{document}